\documentclass[11pt,leqno]{amsart}
\usepackage{amsmath,amsfonts,amssymb,amscd,amsthm,amsbsy,upref}
\numberwithin{equation}{section}
\newtheorem{thm}{Theorem}[section]
\newtheorem{lem}[thm]{Lemma}
\newtheorem{cor}[thm]{Corollary}

\theoremstyle{definition}
\newtheorem{defn}{Definition}[section]

\begin{document}
\title{Quantifications of strictly singular operators and strictly cosingular operators\footnote{Corresponding author: Dongyang Chen}}
\author{Lei Li}
\address{School of Mathematical Sciences and LPMC, Nankai University, Tianjin, 300071, China}
\email{leilee@nankai.edu.cn}
\author{Dongyang Chen}
\address{School of Mathematical Sciences\\ Xiamen University,
Xiamen,361005,China}
\email{cdy@xmu.edu.cn}

\thanks{Dongyang Chen's project was supported by the Natural Science Foundation of Fujian Province of China(No.2015J01026).
Lei Li's research is partly supported by the NSF of China (11301285)}

\begin{abstract}
We investigate possible quantifications of strictly singular operators, $l_{p}$-strictly singular operators, $c_{0}$-strictly singular operators, strictly cosingular operators, $l_{p}$-strictly cosingular operators. We prove quantitative, even strengthening versions of well-known results about relationships of these five classes of operators and compact, weakly compact, unconditionally converging operators.
\end{abstract}

\date{Version: \today}
\maketitle

\baselineskip=18pt 	
\section{Introduction and notations}
This paper is motivated by a large number of recent results on quantitative versions of various theorems and properties of Banach spaces. For example, quantitative versions of Krein's theorem were studied in \cite{FHMZ}, quantitative versions of Eberlein-\u{S}mulyan and Gantmacher theorems were investigated in \cite{AC}, a quantitative version of James' compactness theorem in \cite{CKS}, quantifications of weak sequential completeness and of the Schur property in \cite{KPS} and \cite{KS1}, quantitative Dunford-Pettis property in \cite{KKS}, quantification of the Banach-Saks property in \cite{BKS}, quantification of Pe{\l}czy\'{n}ski's property ($V$) in \cite{Kr} and \cite{Kr1}, quantitative Grothendieck property in \cite{Ben}, etc.

Let us first fix some necessary notations:

If $A$ and $B$ are two nonempty subsets of a Banach space $X$, we set $$d(A,B)=\inf\{\|a-b\|:a\in A,b\in B\},$$$$\widehat{d}(A,B)=\sup\{d(a,B):a\in
A\}.$$ Thus, $d(A,B)$ is the ordinary distance between $A$ and $B$, and $\widehat{d}(A,B)$ is the non-symmetrized Hausdorff distance from $A$ to $B$.

Let $A$ be a bounded subset of a Banach space $X$. The Hausdorff measure of non-compactness of $A$ is defined by
\begin{center}
$\chi(A)=\inf\{\widehat{d}(A,F):F\subset X$ finite$\}$.
\end{center}
Then $\chi(A)=0$ if and only if $A$ is relatively norm compact.

An analogue of Hausdorff measure of non-compactness for measuring weak non-compactness is the de Blasi measure of weak non-compactness
\begin{center}
$\omega(A)=\inf\{\widehat{d}(A,K):\emptyset\neq K\subset X$ is weakly compact $\}.$
\end{center}
Other commonly used quantities measuring weak non-compactness are:
\begin{center}
$wk_{X}(A)=\widehat{d}(\overline{A}^{w^{*}},X),$ where $\overline{A}^{w^{*}}$ denotes the $weak^{*}$ closure of $A$ in $X^{**}$;
\end{center}

\begin{center}
$wck_{X}(A)=\sup\{d(clust_{X^{**}}((x_{n})_{n}),X):(x_{n})_{n}$ is a sequence in $A\}$, where $clust_{X^{**}}((x_{n})_{n})$ is the set of all $weak^{*}$ cluster points of $(x_{n})_{n}$ in $X^{**}$;
\end{center}

\begin{center}
$\gamma_{X}(A)=\sup\{|\lim_{n}\lim_{m}<x^{*}_{m},x_{n}>-\lim_{m}\lim_{n}<x^{*}_{m},x_{n}>|:(x_{n})_{n}$ is a sequence in $A$, $(x^{*}_{m})_{m}$ is a sequence in $B_{X^{*}}$ and all the involved limits exist$\}$.
\end{center}
It follows from \cite[Theorem 2.3]{AC} that
\begin{equation}\label{8}
wck_{X}(A)\leq wk_{X}(A)\leq \gamma_{X}(A)\leq 2wck_{X}(A),
\end{equation}
$$wk_{X}(A)\leq \omega(A).$$
For an operator $T: X\rightarrow Y$, $\chi(T), \omega(T), wk_{Y}(T), wck_{Y}(T), \gamma_{Y}(T)$ will denote $\chi(TB_{X})$,

\noindent $\omega(TB_{X}), wk_{Y}(TB_{X})$, $wck_{Y}(TB_{X})$ and $\gamma_{Y}(TB_{X})$, respectively.

Schauder's theorem states that an operator $T: X\rightarrow Y$ is compact if and only if $T^{*}$ is compact. A quantitative strengthening of Schauder's result was proved by L. S. Gol'den\v{s}te\v{\i}n and A. S. Markus \cite{GM}(also see \cite{KKS}) who established the inequalities as follows:
\begin{equation}\label{79}
\frac{1}{2}\chi(T)\leq \chi(T^{*})\leq 2\chi(T).
\end{equation}
For weak topologies Gantmacher's theorem states that an operator $T: X\rightarrow Y$ is weakly compact if and only if $T^{*}$ is weakly compact.
C. Angosto and B. Cascales(\cite{AC}) established a quantitative version of Gantmacher's theorem:
\begin{equation}\label{205}
\gamma_{Y}(T)\leq \gamma_{X^{*}}(T^{*})\leq 2\gamma_{Y}(T).
\end{equation}
C. Angosto and B. Cascales(\cite[Remark 3.3]{AC}) pointed out that the corresponding quantitative version to (\ref{205}) where $\gamma$ is replaced by $\omega$ fails for general Banach spaces.

Recall that an operator $T:X\rightarrow Y$ is called \textit{unconditionally converging} if $T$ takes weakly unconditionally Cauchy series in $X$ to
unconditionally converging series in $Y$. A well-known characterization of unconditionally converging operators due to A. Pe{\l}czy\'{n}ski \cite{P1} is that an operator $T:X\rightarrow Y$ is unconditionally converging if and only if $T$ does not fix a $c_{0}$-copy. To quantifying this characterization, H. Kruli\v{s}ov\'{a} \cite{Kr} introduced two quantities for an operator $T:X\rightarrow Y$ as follows:
$$uc(T)=\sup\{ca((\sum_{i=1}^{n}Tx_{i})_{n}):(x_{n})_{n}\subseteq X, \sup_{x^{*}\in B_{X^{*}}}\sum_{n=1}^{\infty}|<x^{*},x_{n}>|\leq 1\},$$
\begin{center}
$fix_{c_{0}}(T)=\sup\{(\|U\|\|V\|)^{-1}:M$ is an infinite-dimensional subspace of $X$ for which $T|_{M}$ is an isomorphism and $(T|_{M})^{-1}=UV$ for
some surjective isomorphisms $U:c_{0}\rightarrow M, V:TM\rightarrow c_{0}\}$,
\end{center}
and proved the following inequality:
\begin{equation}\label{206}
\frac{1}{2}uc(T)\leq fix_{c_{0}}(T)\leq uc(T).
\end{equation}

In this paper, we'll concentrate on quantitative, even strengthening versions of classical known results about strictly singular operators, $l_{p}$-strictly singular operators (or $c_{0}$-strictly singular operators), strictly cosingular operators and $l_{p}$-strictly cosingular operators.
Strictly singular operators were introduced by T. Kato \cite{Ka} in connection with the perturbation theory of Fredholm operators. Recall that an
operator $T:X\rightarrow Y$ between Banach spaces is called \textit{strictly singular} if it is not an isomorphism when restricted to any
infinite-dimensional (closed) subspace of $X$; Equivalently, for every $\epsilon>0$ and every infinite-dimensional subspace $M$ of $X$, there is a
$x\in S_{M}$ such that $\|Tx\|<\epsilon$. A. Pe{\l}czy\'{n}ski \cite{P1} introduced the concept of strictly cosingular operators as in a sense dual to strictly singular operators. An operator $T:X\rightarrow Y$ is said to be strictly cosingular provided that for no
infinite-dimensional  Banach space $Z$ there exist surjective operators $R:X\rightarrow Z$ and $S:Y\rightarrow Z$ such that $R=ST$; Equivalently,
there is no infinite-codimensional subspace $N$ of $Y$ such that $Q_{N}T$ is surjective, where $Q_{N}:Y\rightarrow Y/N$ is the quotient map.
Given an operator $T:X\rightarrow Y$ and a Banach space $Z$, say that $T$ is \textit{Z-strictly singular} provided that there is no
infinite-dimensional subspace $M$ of $X$ which is isomorphic to $Z$ for which $T|_{M}$ is an isomorphism; Equivalently, there is no operator
$R:Z\rightarrow X$ for which $TR$ is an isomorphism. Similarly, $T$ is said to be \textit{Z-strictly cosingular} if there is no operator
$S:Y\rightarrow Z$ for which $ST$ is surjective.

Let $X$ be a Banach space, $1\leq p<\infty$ and we denote by $l^{w}_{p}(X)$ the space of all weakly $p$-summable sequences in $X$, endowed with the norm $$\|(x_{n})_{n}\|_{p}^{w}=\sup\{(\sum_{n=1}^{\infty}|<x^{*},x_{n}>|^{p})^{\frac{1}{p}}:x^{*}\in B_{X^{*}}\}, \quad (x_{n})_{n}\in l^{w}_{p}(X).$$
A sequence $(x_{n})_{n}\in l^{w}_{p}(X)$ is \textit{unconditionally $p$-summable} if $$\sup\{(\sum_{n=m}^{\infty}|<x^{*},x_{n}>|^{p})^{\frac{1}{p}}:x^{*}\in B_{X^{*}}\}\rightarrow 0\quad\text{as}\; m\rightarrow \infty.$$
In \cite{CCL}, we say that an operator $T:X\rightarrow Y$ is \textit{unconditionally $p$-converging} if $T$ takes weakly $p$-summable sequences (weakly null sequences for $p=\infty$) to unconditionally $p$-summable sequences (norm null sequences for $p=\infty$). It should be mentioned that unconditionally $p$-converging operators coincide with the $p$-converging operators introduced by J. M. F. Castillo and F. S\'{a}nchez in \cite{CS2} although their original definitions are different. In this paper, we use the terminology unconditionally $p$-converging operators instead of $p$-converging operators. Unconditionally $1$-converging operators are precisely unconditionally converging operators. Given an operator $T: X\rightarrow Y$. We set
\begin{center}
$uc_{p}(T)=\sup\{\limsup_{n}\|Tx_{n}\|: (x_{n})_{n}\in l^{w}_{p}(X), \|(x_{n})_{n}\|_{p}^{w}\leq 1\}.$
\end{center}
An easy verification shows that $uc(T)=uc_{1}(T)$. Obviously, $T: X\rightarrow Y$ is unconditionally $p$-converging if and only if $uc_{p}(T)=0$.

The present paper is organized as follows:

In Section 2, we introduce a quantity $SS(\cdot)$ measuring strict singularity and a quantity $SS_{l_{p}}(\cdot)$ ($SS_{c_{0}}(\cdot)$) measuring $l_{p}$-strict singularity ($c_{0}$-strict singularity) of an operator. First, we prove a quantitative version of implications among three classes of operators-compact, strictly singular and unconditionally converging operators. A classical result due to C. Bessaga and A. Pe{\l}czy\'{n}ski \cite{BP}(also see \cite[Theorem 2.4.10]{AK}) is that for an operator with domain $c_{0}$, compactness,
weak compactness and strict singularity are all equivalent. Section 2 contains a quantitative version of this well-known result. In \cite{P1}, A. Pe{\l}czy\'{n}ski showed that for an operator $T:C(K)\rightarrow X$ the properties of being weakly compact, unconditionally
converging or strictly singular are equivalent. In this section, we give a quantitative version of this main result of \cite{P1} when $K$ is dispersed.

Section 3 contains a quantity $SCS(\cdot)$ measuring strict cosingularity and a quantity $SCS_{l_{p}}(\cdot)$ ($SCS_{c_{0}}(\cdot)$) measuring $l_{p}$-strict cosingularity ($c_{0}$-strict cosingularity) of an operator. We collect quantitative versions of basic relationships between strictly singular operators and strictly cosingular operators. It is elementary that every strictly singular operator on $l_{p}$ or $c_{0}$ is compact. We give a stronger quantitative version of this fact and generalize it. The main result of \cite{P2} is that weak compactness and strict cosingularity are equivalent for operators with range space $L_{1}(\mu)$. In Theorem \ref{3.201}, we prove a quantitative version of this main result and strengthen it. Finally, we establish quantitative versions of equivalence between compact, strictly singular, $l_{2}$-strictly singular, strictly cosingular and $l_{2}$-strictly cosingular operators on the James space.

In order to discuss the connection between the strict singularity of an operator $T$ and that of its adjoint $T^{*}$, R. J. Whitley \cite{W} introduce the concept of subprojective spaces and superprojective spaces. In \cite{OS}, T. Oikhberg and E. Spinu introduced the notion of uniformly subprojective Banach spaces. For the sake of convenience, we refer to call them $\lambda$-subprojective Banach spaces. In Section 4, using the notion of $\lambda$-subprojective spaces, we prove a quantitative version of the main result \cite[Theorem 2.2]{W}. We introduce the notion of $\lambda$-superprojective  spaces as in a sense dual to $\lambda$-subprojective spaces and prove a quantitative version of the duality of \cite[Theorem 2.2]{W}.

Throughout this paper, all Banach spaces are infinite dimensional. An operator will always mean a bounded linear operator. An operator $T:X\rightarrow Y$ is called an \textit{isomorphism} if it has a bounded inverse. If $X$ is a Banach space, we denote by $B_{X}$ its closed unit ball $\{x\in X:
\|x\|\leq 1\}$ and by $S_{X}$ its unit sphere $\{x\in X: \|x\|=1\}$. For a subspace $M$ of $X$, $M^{\perp}=\{x^{*}\in X^{*}: <x^{*},x>=0$ for all
$x\in M\}$. For a subspace $N$ of $X^{*}$, $^\perp N=\{x\in X: <x^{*},x>=0$ for all $x^{*}\in N\}$. $J_{X}:X\rightarrow X^{**}$ denotes the canonical
embedding. Our notation and terminology are standard and we refer the readers to \cite{AK} and \cite{LT} for any unexplained terms.


\section{Quantifying strictly singular operators}
For an operator $T:X\rightarrow Y$, we define the following quantity:

$$SS(T)=\sup_{M\subseteq X, dimM=\infty}\inf_{x\in S_{M}}\|Tx\|,$$
where the supremum is taken over all infinite-dimensional subspaces of $X$. Obviously, $SS(T)=0$ if and only if $T$ is strictly singular.

Given an operator $T:X\rightarrow Y$ and a Banach space $Z$. We set
\begin{center}
$SS_{Z}(T)=\sup\{(\|R\|\|(TR)^{-1}\|)^{-1}:R:Z\rightarrow X$ is an operator for which $TR$ is an isomorphism$\}$.
\end{center}
If there are no such operators $R$'s, we set $SS_{Z}(T)=0$. Hence $SS_{Z}(T)=0$ if and only if $T:X\rightarrow Y$ is $Z$-strictly singular.
Obviously, $SS_{Z}(T)\leq SS(T)$.
A standard argument shows that $SS_{c_{0}}(T)=fix_{c_{0}}(T)$.

It is known that given an operator $T:X\rightarrow Y$, the following implication holds:
\begin{center}
$T$ is compact $\Rightarrow$  $T$ is strictly singular $\Rightarrow$  $T$ is unconditionally converging
\end{center}
We quantify the above implication as in Theorem \ref{3.8}. The quantification means, roughly speaking, to replace implications between some notions by inequalities between certain quantities.

\begin{thm}\label{3.8}
Let $T:X\rightarrow Y$ be an operator. Then $SS_{c_{0}}(T)\leq SS(T)\leq 2\chi(T)$.
\end{thm}
\begin{proof}
We only prove $SS(T)\leq 2\chi(T)$. Let us fix any $0<c<SS(T)$. By the definition of $SS(T)$, there exits an infinite-dimensional subspace $M$ of $X$ such that $\|Tx\|\geq c\|x\|$ for all $x\in M$. Since $M$ is infinite-dimensional, there is a sequence $(x_{n})_{n}$ in $S_{M}$ such that
$$\|x_{n}-x_{m}\|>1, \quad \forall n\neq m.$$
This implies that $$\|Tx_{n}-Tx_{m}\|\geq c\|x_{n}-x_{m}\|>c, \quad \forall n\neq m.$$
Claim: $\widehat{d}(TB_{X},F)\geq \frac{c}{2}$ for every finite subset $F$ of $Y$.

Otherwise, if $\widehat{d}(TB_{X},F)<\frac{c}{2}$, then there exist $y_{0}\in F, x_{m}$ and $x_{n}(m\neq n)$ such that
$\|Tx_{m}-y_{0}\|,\|Tx_{n}-y_{0}\|\leq \frac{c}{2}$. This yields that $\|Tx_{n}-Tx_{m}\|\leq c$, a contradiction.

By Claim, $\chi(T)\geq \frac{c}{2}$. Since $c$ is arbitrary, we get the conclusion.

\end{proof}

Given an operator $T:X\rightarrow Y$. We set $$m(T)=\inf\{\|T|_{M}\|: M\subseteq X, codimM<\infty\},$$
where the infimum is taken over all finite-codimensional subspaces of $X$. It is known that $m(T)=\chi(T^{*})$. We'll appeal several times to this
fact in the sequel.

\begin{thm}\label{3.90}
Let $T:c_{0}\rightarrow X$ be an operator. Then $$\frac{1}{2}\omega(T^{*})=\frac{1}{2}\chi(T^{*})\leq SS_{c_{0}}(T)\leq SS(T)\leq 2\chi(T).$$
\end{thm}
\begin{proof}
$\omega(T^{*})=\chi(T^{*})$ follows from the Schur property.
In view of Theorem \ref{3.8}, we only prove $\chi(T^{*})\leq 2SS_{c_{0}}(T)$.

Suppose that $\chi(T^{*})>0$ and fix any $0<c<\chi(T^{*})$. By induction on finite-codimensional subspaces of $c_{0}$, we obtain a block basic
sequence $(z_{n})_{n}$ with respect to the unit vector basis $(e_{n})_{n}$ of $c_{0}$ such that $\|z_{n}\|\leq 1$ and $\|Tz_{n}\|>c$ for all $n$.
Let $\epsilon>0$. By \cite[Proposition 1.5.4]{AK}, we may assume that $(Tz_{n})_{n}$ is a basic sequence with basis constant $\leq 1+\epsilon$.
Define an operator $R:c_{0}\rightarrow c_{0}$ by $Re_{n}=z_{n}$ for each $n$. Then $\|R\|\leq 1$. For all scalars $a_{1},a_{2},...,a_{m}$ and $m\in
\mathbb{N}$, we have, for each $k=1,2,...,m$
\begin{align*}
c|a_{k}|\leq \|a_{k}Tz_{k}\|&=\|\sum_{j=1}^{k}a_{j}Tz_{j}-\sum_{j=1}^{k-1}a_{j}Tz_{j}\|\\
&\leq \|\sum_{j=1}^{k}a_{j}Tz_{j}\|+\|\sum_{j=1}^{k-1}a_{j}Tz_{j}\|\\
&\leq 2(1+\epsilon)\|\sum_{j=1}^{m}a_{j}Tz_{j}\|,\\
\end{align*}
which yields
\begin{equation}\label{69}
c\sup_{1\leq j\leq m}|a_{j}|\leq 2(1+\epsilon)\|\sum_{j=1}^{m}a_{j}Tz_{j}\|.
\end{equation}
By (\ref{69}), the operator $TR:c_{0}\rightarrow X$ is an isomorphism and $\|(TR)^{-1}\|^{-1}\geq \frac{c}{2(1+\epsilon)}$. By the definition of
$SS_{c_{0}}(T)$, we get $$SS_{c_{0}}(T)\geq \|R\|^{-1}\|(TR)^{-1}\|^{-1}\geq \frac{c}{2(1+\epsilon)}.$$
Since $\epsilon>0$ is arbitrary, $SS_{c_{0}}(T)\geq\frac{c}{2}.$ By the arbitrariness of $c$, we get
$\chi(T^{*})\leq 2SS_{c_{0}}(T)$.
\end{proof}
By interchanging the role of the domain space and the range space of operators in Theorem \ref{3.90}, we obtain a sharp result.
\begin{thm}\label{3.4}
Let $T:X\rightarrow c_{0}$ be an operator. Then $$SS_{c_{0}}(T)=SS(T).$$
\end{thm}
\begin{proof}
It suffices to prove that $SS(T)\leq SS_{c_{0}}(T)$.
As usual, we assume that $SS(T)>0$ and fix any $0<c<SS(T)$. Then there exits an infinite-dimensional subspace $M$ of $X$ such that $\|Tx\|\geq c\|x\|$
for all $x\in M$. Let $\epsilon>0$. By James's $c_{0}$-distortion theorem, there is a sequence $(z_{n})_{n}$ in $B_{TM}$ such that
\begin{equation}\label{63}
(1-\epsilon)\sup_{1\leq k\leq n}|a_{k}|\leq \|\sum_{k=1}^{n}a_{k}z_{k}\|\leq \sup_{1\leq k\leq n}|a_{k}|,
\end{equation}
for all scalars $a_{1},a_{2},...,a_{n}$ and all $n\in \mathbb{N}$. Define an operator $S:c_{0}\rightarrow X$ by $Se_{k}=T^{-1}z_{k}(k=1,2,...)$.
According to (\ref{63}), we have
$$\|\sum_{k=1}^{n}a_{k}T^{-1}z_{k}\|\leq \frac{1}{c}\|\sum_{k=1}^{n}a_{k}z_{k}\|\leq \frac{1}{c}\sup_{1\leq k\leq n}|a_{k}|,$$
for all scalars $a_{1},a_{2},...,a_{n}$ and all $n\in \mathbb{N}$. This yields $\|S\|\leq \frac{1}{c}$.

Moreover, the left side of inequality (\ref{63}) implies that the operator $TS:c_{0}\rightarrow c_{0}$ is an isomorphism and $\|(TS)^{-1}\|^{-1}\geq
(1-\epsilon)$. Thus, we get $$SS_{c_{0}}(T)\geq (\|S\|)^{-1}\|(TS)^{-1}\|^{-1}\geq c(1-\epsilon).$$
Letting $\epsilon\rightarrow 0$, we get $$SS_{c_{0}}(T)\geq c.$$
Since $c$ is arbitrary, we finish the proof.
\end{proof}

\begin{lem}\cite{CCL}\label{3.6}
Let $X$ be a Banach space and $(x_{n})_{n}$ be a weakly null sequence in $B_{X}$. Let $\epsilon>0$ be such that $\|x_{n}\|>\epsilon$ for all $n\in
\mathbb{N}$. Then, for every $\delta>0$, there is a subsequence $(x_{k_{n}})_{n}$ of $(x_{n})_{n}$ such that $\|x_{k_{n}}-x_{k_{m}}\|\geq
\epsilon-\delta(n\neq m,n,m=1,2,...)$.
\end{lem}
\begin{thm}\label{3.9}
Let $K$ be a dispersed compact Hausdorff space and $T:C(K)\rightarrow X$ an operator. Then $$\frac{1}{8\pi}\omega(T^{*})\leq SS_{c_{0}}(T)\leq
SS(T)\leq 2\omega(T^{*}).$$ In the real case the constant $\pi$ can be replaced by $2$.
\end{thm}
\begin{proof}
We may assume that $\|T\|=1$.

Step 1. $\omega(T^{*})\leq 8\pi SS_{c_{0}}(T)$.

Let $A=T^{*}B_{X^{*}}$. Suppose that $\omega(A)>0$ and fix any $0<c<\omega{(A)}$. According to \cite[Proposition 5.2]{KS}, there exist a sequence $(U_{k})_{k}$ of pairwise disjoint open subsets of $K$ and a sequence $(\mu_{k})_{k}$ in $A$ such that $\mu_{k}(U_{k})>\frac{c}{\pi}$ for all $k$. Let $\epsilon>0$.
By the regularity of $\mu_{k}$, there is a compact subset $F_{k}$ of $U_{k}$ such that $|\mu_{k}|(U_{k}\setminus F_{k})<\epsilon$ for each $k$.
For each $k$, by Urysohn's Lemma, there is a $f_{k}\in C(K)$ with $0\leq f_{k}\leq 1$ such that $f_{k}=1$ on $F_{k}$ and $f_{k}=0$ on $K\setminus
U_{k}$. Since the sequence $(U_{k})_{k}$ is pairwise disjoint, we get
$\|\sum_{k=1}^{n}a_{k}f_{k}\|=\sup_{1\leq k\leq n}|a_{k}|$, for all scalars $a_{1},a_{2},...,a_{n}$ and all $n\in \mathbb{N}$. Define an operator
$S:c_{0}\rightarrow C(K)$ by $Se_{k}=f_{k}(k=1,2,...)$. Then $S$ is an isometric embedding. Claim: $\|TSe_{k}\|>\frac{c}{\pi}-\epsilon$ for each $k$.

Indeed, let $\mu_{k}=T^{*}y^{*}_{k} (y^{*}_{k}\in B_{Y^{*}})$. Then
\begin{align*}
\|TSe_{k}\|&\geq |<y^{*}_{k},TSe_{k}>|\\
&=|<\mu_{k},f_{k}>|\\
&=|\int_{U_{k}}1d\mu_{k}-\int_{U_{k}}(1-f_{k})d\mu_{k}|\\
&>\frac{c}{\pi}-|\int_{U_{k}\setminus F_{k}}(1-f_{k})d\mu_{k}|\\
&\geq \frac{c}{\pi}-\int_{U_{k}\setminus F_{k}}1d|\mu_{k}|\\
&>\frac{c}{\pi}-\epsilon.\\
\end{align*}
Let $\delta>0$ be arbitrary. It follows from Lemma \ref{3.6} that there is a subsequence $(TSe_{k_{n}})_{n}$ of $(TSe_{k})_{k}$ such that
$\|TSe_{k_{n}}-TSe_{k_{m}}\|\geq\frac{c}{\pi}-\epsilon-\delta$ for all $n\neq m$. This yields that $\chi(TS)\geq
\frac{\frac{c}{\pi}-\epsilon-\delta}{2}$. Since $\delta>0$ is arbitrary, we get $\chi(TS)\geq \frac{\frac{c}{\pi}-\epsilon}{2}$.
Theorem \ref{3.90} and (\ref{79}) ensure that
$$\frac{\frac{c}{\pi}-\epsilon}{4}\leq 2SS_{c_{0}}(TS)\leq 2SS_{c_{0}}(T).$$
By the arbitrariness of $\epsilon>0$, we have $\frac{c}{4\pi}\leq 2SS_{c_{0}}(T).$ Since $c$ is arbitrary, we conclude Step 1.

Step 2. $SS(T)\leq 2\omega(T^{*}).$

Assume that $SS(T)>0$ and fix any $0<c<SS(T)$. Then there exits an infinite-dimensional subspace $M$ of $C(K)$ such that $\|Tf\|\geq c\|f\|$ for all
$f\in M$. Since $M$ is infinite-dimensional, there is a sequence $(f_{n})_{n}$ in $S_{M}$ such that
$$\|f_{n}-f_{m}\|>1, \quad \forall n\neq m.$$
This implies that $$\|Tf_{n}-Tf_{m}\|\geq c\|x_{n}-x_{m}\|>c, \quad \forall n\neq m.$$
Since $K$ is dispersed, it follows from Main theorem in \cite{PS} that $C(K)$ contains no isomorphic copy of $l_{1}$. By Rosenthal's $l_{1}$-theorem,
we may assume that the sequence $(f_{n})_{n}$ is weakly Cauchy. Let $g_{n}=\frac{f_{n}-f_{n+1}}{2}$. Then $(g_{n})_{n}$ is weakly null and
$\|Tg_{n}\|\geq \frac{c}{2}$ for all $n\in \mathbb{N}$. Again by \cite[Proposition 5.2]{KS}, we get $\frac{c}{2}\leq\omega(T^{*})$. The arbitrariness
of $c$ yields the conclusion.
\end{proof}

\begin{thm}
Let $2<p<\infty$ and $T:L_{p}\rightarrow L_{p}$ an operator. Then $$SS(T)=\max\{SS_{l_{p}}(T),SS_{l_{2}}(T)\}.$$
\end{thm}
\begin{proof}
Let $0<c<SS(T)$. Then there exists an infinite-dimensional subspace $M$ of $L_{p}$ such that $\|Tf\|\geq c\|f\|$ for all $f\in M$.

Case 1. $M$ is isomorphic to $l_{2}$ and complemented in $L_{p}$.

Let $\epsilon>0$. By \cite{HOS} (or \cite[Theorem 1.3]{Als}), there exists an isomorphism $R:l_{2}\rightarrow L_{p}$ such that $R(l_{2})\subseteq M$ and $\|R\|\leq 1+\epsilon, \|R^{-1}\|\leq \frac{1}{1-\epsilon}$. This yields that $\|TRz\|\geq c(1-\epsilon)\|z\|$ for each $z\in l_{2}$. Thus we get $$SS_{l_{2}}(T)\geq \|R\|^{-1}\|(TR)^{-1}\|^{-1}\geq \frac{c(1-\epsilon)}{1+\epsilon},$$ which implies $SS_{l_{2}}(T)\geq c$ since $\epsilon$ is arbitrary.

Case 2. For every $\epsilon>0$, $M$ contains a subspace $N$ which is $(1+\epsilon)$-complemented in $L_{p}$ and satisfies $d(N,l_{p})<1+\epsilon$. Take a surjective isomorphism $R:l_{p}\rightarrow N$ such that $\|R\|=1, \|R^{-1}\|\leq 1+\epsilon$. Then $\|TRz\|\geq \frac{c}{(1+\epsilon)}\|z\|$ for $z\in l_{p}$. Thus $$SS_{l_{p}}(T)\geq \|R\|^{-1}\|(TR)^{-1}\|^{-1}\geq \frac{c}{1+\epsilon}.$$
Letting $\epsilon\rightarrow 0$, we get $SS_{l_{p}}(T)\geq c$.

In both cases, we have $\max\{SS_{l_{p}}(T),SS_{l_{2}}(T)\}\geq c$. It follows that $SS(T)\leq\max\{SS_{l_{p}}(T),SS_{l_{2}}(T)\}.$
The proof is completed.

\end{proof}

We'll need a quantitative version of the Bessaga-Pe{\l}czy\'{n}ski Selection Principle. More specifically, we need small uniform bounds on the
equivalence constant and projection constant. Its proof is identical to the standard gliding hump arguments (see \cite{AK} or \cite{D}).

\begin{thm}\label{3.22}
Let $(x_{n})_{n}$ be a basis for a Banach space $X$ and $(x^{*}_{n})_{n}$ be the sequence of coefficient functionals. If $(y_{n})_{n}$ is a
semi-normalized weakly null sequence in $X$, then, for every $\epsilon>0$, there exist a subsequence $(y_{k_{n}})_{n}$ of $(y_{n})_{n}$ and a
(skipped) block basic sequence $(z_{n})_{n}$ with respect to $(x_{n})_{n}$ such that
$$(1-\epsilon)\|\sum_{i=1}^{n}a_{i}z_{i}\|\leq \|\sum_{i=1}^{n}a_{i}y_{k_{i}}\|\leq (1+\epsilon)\|\sum_{i=1}^{n}a_{i}z_{i}\|,$$
for all scalars $a_{1},a_{2},...,a_{n}$ and all $n\in \mathbb{N}$. If every semi-normalized (skipped) block basic sequence with respect to
$(x_{n})_{n}$ is $C$-complemented in $X$ (where the constant $C$ depends only on $X$), then $\overline{span}\{y_{k_{n}}:n\in \mathbb{N}\}$ is
$C\cdot\frac{1+\epsilon}{1-\epsilon}$-complemented in $X$.
\end{thm}
Finally, we give a quantitative version of \cite[Theorem 25]{Cas}.
\begin{thm}
Let $1<p<2$ and $T:X\rightarrow L_{p}$ an operator. Then $$\frac{1}{4(p^{*}-1)B_{p^{*}}}uc_{2}(T)\leq SS_{l_{2}}(T)\leq uc_{2}(T),$$ where $B_{p^{*}}$ is the Khintchine's constant.
\end{thm}
\begin{proof}
The second inequality of this theorem is straightforward. We only prove the first inequality.

Let $0<c<uc_{2}(T)$. Then there exists a weakly $2$-summable sequence $(x_{n})_{n}$ in $X$ with $\|(x_{n})_{n}\|_{2}^{w}\leq 1$ such that $\|Tx_{n}\|>c$ for all $n$. Let $\epsilon>0$. By passing to subsequences, we may assume that $(Tx_{n})_{n}$ is a basic sequence with the basis constant $\leq 1+\epsilon$. Let $Y=\overline{span}\{Tx_{n}:n\in \mathbb{N}\}$. Let $(y^{*}_{n})_{n}$ be the biorthogonal functionals associated to $(Tx_{n})_{n}$. Then $\|y^{*}_{n}\|\leq \frac{2(1+\epsilon)}{c}$ for all $n$. Let $f_{n}$ be the norm-preserving extension of $y^{*}_{n}$ to the whole space $L_{p}$. Set $K=\frac{2(1+\epsilon)}{c}$. Then there exists a metric $d$ on $B_{L_{p^{*}}}(0,K)$ such that the $weak^{*}$-topology agrees with the $d$-topology. Thus all $weak^{*}$-cluster points of $(f_{n})_{n}$ are in $F=B_{L_{p^{*}}}(0,K)\cap Y^{\perp}$. It is easy to see that $\lim_{n\rightarrow \infty}d(f_{n},F)=0$. We pick a sequence $(g_{n})_{n}$ in $F$ such that $\lim_{n\rightarrow \infty}d(f_{n},g_{n})=0$. Let $u^{*}_{n}=f_{n}-g_{n}$. Then $(u^{*}_{n})_{n}$ is weakly null, biorthogonal to $(Tx_{n})_{n}$, $\|u^{*}_{n}\|\leq \frac{4(1+\epsilon)}{c}$ and $u^{*}_{n}|_{Y}=y^{*}_{n}$ for all $n$. Let $(h_{n})_{n}$ be the Haar basis for $L_{p^{*}}$ with the unconditional constant $p^{*}-1$. According to Theorem \ref{3.22}, $(u^{*}_{n})_{n}$ admits a subsequence, which is still denoted by $(u^{*}_{n})_{n}$, and a block basic sequence $(z^{*}_{n})_{n}$ with respect to $(h_{n})_{n}$ such that
\begin{equation}\label{201}
(1-\epsilon)\|\sum_{i=1}^{n}a_{i}z^{*}_{i}\|\leq \|\sum_{i=1}^{n}a_{i}u^{*}_{i}\|\leq (1+\epsilon)\|\sum_{i=1}^{n}a_{i}z^{*}_{i}\|,
\end{equation}
for all scalars $a_{1},a_{2},...,a_{n}$ and all $n\in \mathbb{N}$.

By (\ref{201}), we see that $\|z^{*}_{i}\|\leq \frac{4(1+\epsilon)}{c(1-\epsilon)}$ for all $i$. By inequality (1.9) in \cite{AO}, we get
\begin{equation}\label{202}
\|\sum_{i=1}^{n}a_{i}z^{*}_{i}\|\leq (p^{*}-1)B_{p^{*}}\frac{4(1+\epsilon)}{c(1-\epsilon)}(\sum_{i=1}^{n}|a_{i}|^{2})^{\frac{1}{2}},
\end{equation}
for all scalars $a_{1},a_{2},...,a_{n}$ and all $n\in \mathbb{N}$.

Combining (\ref{201}) with (\ref{202}), we get
\begin{equation}\label{203}
\|\sum_{i=1}^{n}a_{i}y^{*}_{i}\|\leq\|\sum_{i=1}^{n}a_{i}u^{*}_{i}\|\leq (p^{*}-1)B_{p^{*}}\frac{4(1+\epsilon)^{2}}{c(1-\epsilon)}(\sum_{i=1}^{n}|a_{i}|^{2})^{\frac{1}{2}},
\end{equation}
for all scalars $a_{1},a_{2},...,a_{n}$ and all $n\in \mathbb{N}$.

Since $Y$ is reflexive, $(Tx_{n})_{n}$ is shrinking and hence $(y^{*}_{n})_{n}$ forms a basis for $Y^{*}$.
Inequality (\ref{203}) implies that the operator $R:l_{2}\rightarrow Y^{*}$ defined by $Re_{n}=y^{*}_{n}(n\in \mathbb{N})$ is well-defined and $\|R\|\leq (p^{*}-1)B_{p^{*}}\frac{4(1+\epsilon)^{2}}{c(1-\epsilon)}$. Taking the adjoint, we see that $R^{*}Tx_{n}=e_{n}$ for all $n$. Thus,
for all scalars $a_{1},a_{2},...,a_{n}$, we have
\begin{equation}\label{204}
(\sum_{i=1}^{n}|a_{i}|^{2})^{\frac{1}{2}}\leq (p^{*}-1)B_{p^{*}}\frac{4(1+\epsilon)^{2}}{c(1-\epsilon)}\|\sum_{i=1}^{n}a_{i}Tx_{i}\|.
\end{equation}
Define an operator $S:l_{2}\rightarrow X$ by $Se_{n}=x_{n}(n\in \mathbb{N})$. Then $\|S\|=\|(x_{n})_{n}\|_{2}^{w}\leq 1$. By this fact together with (\ref{204}), we get
$$SS_{l_{2}}(T)\geq \|S\|^{-1}\|(TS)^{-1}\|^{-1}\geq \frac{c(1-\epsilon)}{4(p^{*}-1)B_{p^{*}}(1+\epsilon)^{2}}.$$
Letting $\epsilon\rightarrow 0$, we get $SS_{l_{2}}(T)\geq \frac{c}{4(p^{*}-1)B_{p^{*}}}.$ Since $c$ is arbitrary, we are done.

\end{proof}

\section{Quantifying strictly cosingular operators}

Given a surjective operator $T:X\rightarrow Y$. We set $$\delta(T)=\sup\{\delta>0: \delta B_{Y}\subseteq TB_{X}\}.$$
Given an operator $T:X\rightarrow Y$. We define a quantity as follows:
\begin{center}
$SCS(T)=\sup\{\delta(Q_{N}T): N$ is an infinite-codimensional subspace of $Y$ such that $Q_{N}T$ is surjective$\}$.
\end{center}
If there are no such subspaces $N$'s, we set $SCS(T)=0$. Thus $T$ is strictly cosingular if and only if $SCS(T)=0$. A routine argument shows:
\begin{center}
$SCS(T)=\sup\{\frac{\delta(ST)}{\|S\|}: Z$ Banach space and operator $S:Y\rightarrow Z$ such that $ST$ is surjective$\}$,
\end{center}
where the supremum is taken over all Banach spaces $Z$ and all operators $S:Y\rightarrow Z$ such that $ST$ is surjective.

Given an operator $T:X\rightarrow Y$ and a Banach space $Z$. We set
\begin{center}
$SCS_{Z}(T)=\sup\{\frac{\delta(ST)}{\|S\|}: S:Y\rightarrow Z$ is an operator such that $ST$ is surjective$\}$,
\end{center}
where the supremum is taken over all operators $S:Y\rightarrow Z$ such that $ST$ is surjective.
If there are no such operators $S$'s, we set $SCS_{Z}(T)=0$. Thus $T$ is $Z$-strictly cosingular if and only if $SCS_{Z}(T)=0$.

For an operator $T:X\rightarrow Y$, the following implication holds:
\begin{center}
$T$ is compact $\Rightarrow$  $T$ is strictly cosingular
\end{center}
We quantify this implication as follows:

\begin{thm}\label{3.14}
Let $T:X\rightarrow Y$ be an operator. Then $SCS(T)\leq 2\chi(T^{*}).$
\end{thm}
\begin{proof}
Suppose that $SCS(T)>0$ and fix arbitrary $0<c<SCS(T)$. Then there is an infinite-codimensional subspace $N$ of $Y$ such that $c\cdot B_{Y/N}\subseteq
Q_{N}TB_{X}$. This implies that $\|T^{*}y^{*}\|\geq c\|y^{*}\|$ for all $y^{*}\in N^{\perp}$. Since $N^{\perp}$ is infinite-dimensional, there exists
a sequence $(y^{*}_{n})_{n}$ in $S_{N^{\perp}}$ such that $\|y^{*}_{m}-y^{*}_{n}\|>1$ for each $m\neq n$. Therefore, for each $m\neq n$,
$\|T^{*}y^{*}_{m}-T^{*}y^{*}_{n}\|\geq c\|y^{*}_{m}-y^{*}_{n}\|>c$. As in the proof of Theorem \ref{3.8}, we get $\chi(T^{*})\geq \frac{c}{2}$. By the
arbitrariness of $c$, the conclusion follows.

\end{proof}

\begin{thm}\label{3.11}
Let $T:X\rightarrow Y$ be an operator. Then
\item[(1)]$SCS(T)\leq SS(T^{*})$;
\item[(2)]$SS(T)\leq SCS(T^{*}).$
\end{thm}
\begin{proof}
(1). Suppose that $SCS(T)>0$ and fix any $0<c<SCS(T)$. Then there is an infinite-codimensional subspace $N$ of $Y$ such that $c\cdot B_{Y/N}\subseteq
Q_{N}TB_{X}$. This yields that $\|T^{*}y^{*}\|\geq c\|y^{*}\|$ for all $y^{*}\in N^{\perp}$. Thus $c\leq SS(T^{*})$. The arbitrariness of $c$
concludes the proof.

(2). Assume that $SS(T)>0$ and fix any $0<c<SS(T)$. Then there is an infinite-dimensional subspace $M$ of $X$ such that $\|Tx\|\geq c\|x\|$ for all
$x\in M$. It is easy to verify that $c\cdot B_{M^{*}}\subseteq i^{*}_{M}T^{*}B_{Y^{*}}$, where $i_{M}:M\rightarrow X$ is the inclusion map. Therefore
$c\cdot B_{X^{*}/M^{\perp}}\subseteq Q_{M^{\perp}}T^{*}B_{Y^{*}}$, which yields that $c\leq SCS(T^{*})$.
Since $c$ is arbitrary, we get the conclusion.

\end{proof}

\begin{cor}\label{3.12}
Let $T:X\rightarrow Y$ be an operator and $X$ be reflexive. Then
\item[(1)]$SCS(T)=SS(T^{*})$;
\item[(2)]$SS(T)=SCS(T^{*}).$
\end{cor}

\begin{thm}\label{3.15}
Let $X=l_{p}(1<p<\infty)$ or $c_{0}$ and $T:X\rightarrow X$ an operator. Then
\item[(1)]$\chi(T^{*})\leq SS_{X}(T)\leq SS(T)\leq 2\chi(T);$
\item[(2)]$\chi(T^{*})\leq SCS_{X}(T)\leq SCS(T)\leq 2\chi(T^{*}).$
\end{thm}
\begin{proof}
Fix arbitrary number $c<\chi(T^{*})$. Let $\epsilon>0$. Then there is a block basic sequence $(x_{n})_{n}$ with respect to the unit vector basis of
$X$ such that $\|x_{n}\|\leq 1$ and $\|Tx_{n}\|\geq c-\epsilon$ for all $n\in \mathbb{N}$. By Theorem \ref{3.22}, there exist a subsequence
$(x_{k_{n}})_{n}$ of $(x_{n})_{n}$ and a block basic sequence $(z_{n})_{n}$ with respect to the unit vector basis $(e_{n})_{n}$ of $X$ such that
\begin{equation}\label{84}
(1-\epsilon)\|\sum_{i=1}^{n}a_{i}z_{i}\|\leq \|\sum_{i=1}^{n}a_{i}Tx_{k_{i}}\|\leq (1+\epsilon)\|\sum_{i=1}^{n}a_{i}z_{i}\|,
\end{equation}
for all scalars $a_{1},a_{2},...,a_{n}$ and all $n\in \mathbb{N}$. Moreover, $\overline{span}\{Tx_{k_{n}}:n\in \mathbb{N}\}$ is
$\frac{1+\epsilon}{1-\epsilon}$-complemented in $X$.

\noindent (1). Define an operator $S:X\rightarrow X$ by $Se_{n}=x_{k_{n}}(n=1,2,...)$. Then $\|S\|\leq 1$. By (\ref{84}), we get
$\|z_{n}\|\geq\frac{c-\epsilon}{1+\epsilon}$ for each $n\in \mathbb{N}$. Again by (\ref{84}), we obtain
\begin{equation}\label{65}
\|\sum_{i=1}^{n}a_{i}Tx_{k_{i}}\|\geq
(1-\epsilon)\|\sum_{i=1}^{n}a_{i}z_{i}\|\geq(1-\epsilon)\frac{c-\epsilon}{1+\epsilon}\|\sum_{i=1}^{n}a_{i}e_{i}\|,
\end{equation}
for all scalars $a_{1},a_{2},...,a_{n}$ and all $n\in \mathbb{N}$. Inequality (\ref{65}) implies that the operator $TS:X\rightarrow X$ is an
isomorphism and $\|(TS)^{-1}\|^{-1}\geq(1-\epsilon)\frac{c-\epsilon}{1+\epsilon}$. Thus $$SS_{X}(T)\geq
\|S\|^{-1}\|(TS)^{-1}\|^{-1}\geq(1-\epsilon)\frac{c-\epsilon}{1+\epsilon}.$$
Letting $\epsilon\rightarrow 0$, we get $SS_{X}(T)\geq c$. The arbitrariness of $c$ yields $\chi(T^{*})\leq SS_{X}(T)$.

\noindent (2). Define an operator $U:\overline{span}\{Tx_{k_{n}}:n\in \mathbb{N}\}\rightarrow X$ by $UTx_{k_{n}}=e_{n}(n=1,2,...)$. By (\ref{84}), we
get $\|U\|\leq\frac{1+\epsilon}{(1-\epsilon)(c-\epsilon)}$. Let $R=UP$, where $P$ is a projection from $X$ onto $\overline{span}\{Tx_{k_{n}}:n\in
\mathbb{N}\}$ with $\|P\|\leq \frac{1+\epsilon}{1-\epsilon}$. By the definition of $R$, it is easy to verify that $B_{X}\subseteq RTB_{X}$.
Thus $$SCS_{X}(T)\geq\frac{\delta(RT)}{\|R\|}\geq \frac{(1-\epsilon)^{2}(c-\epsilon)}{(1+\epsilon)^{2}}.$$
Letting $\epsilon\rightarrow 0$, we get $SCS_{X}(T)\geq c$. Since $c$ is arbitrary, we get $\chi(T^{*})\leq SCS_{X}(T)$.
\end{proof}
For $p=1$, we have the following result.

\begin{thm}\label{3.100}
Let $T:X\rightarrow l_{1}$ be an operator. Then
\item[(1)]$\frac{1}{2}\chi(T)\leq SS_{l_{1}}(T)\leq SS(T)\leq 2\chi{(T)}$;
\item[(2)]$\frac{1}{2}\chi(T)\leq SCS_{l_{1}}(T)\leq SCS(T)\leq 2\chi{(T^{*})}.$
\end{thm}
\begin{proof}
Suppose that $\chi(T)>0$ and fix any $0<c<\chi(T)$. Then there is a sequence $(x_{n})_{n}$ in $B_{X}$ such that $\|Tx_{n}-Tx_{m}\|>c$ for all $n\neq
m$. By passing to subsequences, we may assume that $\lim_{n\rightarrow \infty}(Tx_{n})(k)$ exists for each $k\in \mathbb{N}$. Let $\eta>0$ and
$(\epsilon_{n})_{n}$ be a sequence of positive numbers. We can choose two increasing sequences $(p_{i})_{i},(q_{i})_{i}$ of natural numbers such that
\begin{equation}\label{85}
\sum_{k=1}^{q_{i}}|(Ty_{i})(k)|+\sum_{k=q_{i+1}+1}^{\infty}|(Ty_{i})(k)|<\eta\epsilon_{i},i=1,2,...
\end{equation}
where $y_{i}=x_{p_{2i}}-x_{p_{2i+1}}(i=1,2,...)$.

Let $z_{i}=\sum_{k=q_{i}+1}^{q_{i+1}}(Ty_{i})(k)e_{k}$. It follows from (\ref{85}) that $\|Ty_{i}-z_{i}\|<\eta\epsilon_{i}$ for all $i$.

\noindent (1). Take $\eta=\frac{c}{2}$. Let $\epsilon>0$ be arbitrary. We set $\epsilon_{n}=\frac{\epsilon}{2^{n}}(n=1,2,...).$ Then, for all scalars
$a_{1},a_{2},...,a_{n}$ and all $n\in \mathbb{N}$, we have
\begin{align*}
\|\sum_{i=1}^{n}a_{i}Ty_{i}\|&\geq \|\sum_{i=1}^{n}a_{i}z_{i}\|-\sum_{i=1}^{n}|a_{i}|\|Ty_{i}-z_{i}\|\\
&\geq \sum_{i=1}^{n}|a_{i}|(c-\frac{c}{2}\frac{\epsilon}{2^{i}})-\sum_{i=1}^{n}|a_{i}|\frac{c}{2}\frac{\epsilon}{2^{i}}\\
&=c\sum_{i=1}^{n}(1-\frac{\epsilon}{2^{i}})|a_{i}|\\
&\geq c(1-\epsilon)\sum_{i=1}^{n}|a_{i}|.\\
\end{align*}
Define an operator $R:l_{1}\rightarrow X$ by $Re_{n}=y_{n}(n\in \mathbb{N})$. Then $\|R\|\leq 2$.
Moreover, the operator $TR:l_{1}\rightarrow l_{1}$ is an isomorphism and $\|(TR)^{-1}\|^{-1}\geq c(1-\epsilon)$. This together with $\|R\|\leq 2$ yields $SS_{l_{1}}(T)\geq\frac{c(1-\epsilon)}{2}$. Letting $\epsilon\rightarrow 0$, we get $SS_{l_{1}}(T)\geq\frac{c}{2}$.
Since $c$ is arbitrary, we get $\frac{\chi(T)}{2}\leq SS_{l_{1}}(T).$

\noindent (2). Take $\eta=c$. Let $\epsilon>0$ be arbitrary. We set $\epsilon_{n}=\frac{\frac{\epsilon}{2^{n}}}{1+\frac{\epsilon}{2^{n}}}(n=1,2,...).$
For each $i\in \mathbb{N}$, choose $(\eta_{k})_{k=q_{i}+1}^{q_{i+1}}$ with $\sup_{q_{i}+1\leq k\leq q_{i+1}}|\eta_{k}|=1$ such that
$$\sum_{k=q_{i}+1}^{q_{i+1}}\eta_{k}(Ty_{i})(k)=\sum_{k=q_{i}+1}^{q_{i+1}}|(Ty_{i})(k)|=\|z_{i}\|.$$
Set $z^{*}_{i}=\sum_{k=q_{i}+1}^{q_{i+1}}\lambda_{k}e^{*}_{k}$, where $\lambda_{k}=\frac{\eta_{k}}{\|z_{i}\|}$ and $(e^{*}_{k})_{k=1}^{\infty}$ is the
unit vector basis of $c_{0}$. Then $(z^{*}_{i})_{i}$ is the coefficient functionals of $(z_{i})_{i}$ and $\|z^{*}_{i}\|=\frac{1}{\|z_{i}\|}(i\in
\mathbb{N})$. Let $P$ be a norm one projection from $l_{1}$ onto $\overline{span}\{z_{i}:i\in \mathbb{N}\}$. An easy computation shows
$$\sum_{i=1}^{\infty}\|z^{*}_{i}\|\|Ty_{i}-z_{i}\|\leq \sum_{i=1}^{\infty}\frac{\eta\epsilon_{i}}{c-\eta\epsilon_{i}}=\epsilon.$$
Define an operator $R:l_{1}\rightarrow l_{1}$ by $$Rx=x-Px+\sum_{i=1}^{\infty}<z^{*}_{i},Px>Ty_{i}, \quad x\in l_{1}.$$
Then $\|R-I_{l_{1}}\|\leq \epsilon$ and hence $R^{-1}$ exists with $\|R^{-1}\|\leq \frac{1}{1-\epsilon}$. Let $Q=RPR^{-1}$. Then $Q$ is a projection
from $l_{1}$ onto $\overline{span}\{Ty_{i}:i\in \mathbb{N}\}$. An argument similar to (1) shows that
$$\|\sum_{i=1}^{n}a_{i}Ty_{i}\|\geq c(1-\epsilon)\sum_{i=1}^{n}|a_{i}|,$$ for all scalars $a_{1},a_{2},...,a_{n}$ and all $n\in \mathbb{N}$.

Define operators $U:\overline{span}\{Ty_{i}:i\in \mathbb{N}\}\rightarrow l_{1}$ by $UTy_{i}=e_{i}(i\in \mathbb{N})$ and $S:l_{1}\rightarrow l_{1}$ by
$S=UQ$. Thus $$\|S\|\leq \|U\|\|Q\|\leq \frac{1+\epsilon}{c(1-\epsilon)^{2}}.$$ By the definition of $S$, we get $B_{l_{1}}\subseteq 2STB_{X}.$
Finally, we have $$SCS_{l_{1}}(T)\geq \frac{\delta(ST)}{\|S\|}\geq \frac{1}{2}\frac{c(1-\epsilon)^{2}}{1+\epsilon}.$$
Letting $\epsilon\rightarrow 0$, we get $SCS_{l_{1}}(T)\geq \frac{c}{2}.$ The arbitrariness of $c$ concludes (2).
\end{proof}
The following theorem is a quantitative version of \cite[Proposition 1]{P2}.
\begin{thm}\label{3.13}
Let $T:X\rightarrow Y$ be an operator. Then $$SCS_{l_{1}}(T)\leq SS_{c_{0}}(T^{*})\leq 8SCS_{l_{1}}(T).$$
\end{thm}
\begin{proof}
Step 1. $SCS_{l_{1}}(T)\leq SS_{c_{0}}(T^{*})$.

Fix any $0<c<SCS_{l_{1}}(T)$. Then there is an operator $S:Y\rightarrow l_{1}$ such that $\|S\|=1$ and $c\cdot B_{l_{1}}\subseteq STB_{X}$.
Thus $\|T^{*}S^{*}z\|\geq c\|z\|,$ for every $z\in c_{0}$. This yields $$SS_{c_{0}}(T^{*})\geq
\|S^{*}|_{c_{0}}\|^{-1}\|(T^{*}S^{*}|_{c_{0}})^{-1}\|^{-1}\geq c.$$
It follows from the arbitrariness of $c$ that $SCS_{l_{1}}(T)\leq SS_{c_{0}}(T^{*})$.

Step 2. $SS_{c_{0}}(T^{*})\leq 8SCS_{l_{1}}(T).$

Let $0<c<SS_{c_{0}}(T^{*})$ be arbitrary. It suffices to show that $c\leq 8SCS_{l_{1}}(T).$ Then there is an operator $V:c_{0}\rightarrow Y^{*}$ with
$\|V\|=1$ such that $\|T^{*}Vz\|\geq c\|z\|$ for all $z\in c_{0}$. An argument similar to Theorem \ref{3.8} shows that $\chi(T^{*}V)\geq
\frac{c}{2}$.
Let $U=V^{*}J_{Y}T:X\rightarrow l_{1}$. It is easy to check that $U^{*}|_{c_{0}}=T^{*}V$. Thus, by (\ref{79}), we have
$$\chi(U)\geq \frac{1}{2}\chi(U^{*})\geq \frac{1}{2}\chi(T^{*}V)\geq \frac{c}{4}.$$
Let $\epsilon, \delta>0$ are arbitrary. Applying the argument of Theorem \ref{3.100} (2) to $\eta=\frac{c}{4}-\epsilon$, we get an operator
$S:l_{1}\rightarrow l_{1}$ such that $\|S\|\leq \frac{1+\delta}{(\frac{c}{4}-\epsilon)(1-\delta)^{2}}$ and $B_{l_{1}}\subseteq 2SUB_{X}$.
Let $R=SV^{*}J_{Y}:Y\rightarrow l_{1}$. Then $RT=SU, \|R\|\leq \|S\|$ and hence
$$SCS_{l_{1}}(T)\geq \frac{\delta(RT)}{\|R\|}\geq \frac{1}{2}\frac{(\frac{c}{4}-\epsilon)(1-\delta)^{2}}{1+\delta}.$$ Letting $\epsilon\rightarrow 0$
and $\delta\rightarrow 0$, we get $SCS_{l_{1}}(T)\geq \frac{c}{8}.$ The proof is completed.

\end{proof}
We need the following elementary lemma.
\begin{lem}\cite{CCL}\label{5.4}
Let $X$ be a closed subspace of a Banach space $Y$ and let $A$ be a bounded subset of $X$. Then
$$wk_{Y}(A)\leq wk_{X}(A)\leq 2wk_{Y}(A).$$
\end{lem}
It is worth mentioning that the constant 2 in the right inequality of Lemma \ref{5.4} is optimal. Indeed, let $X=c_{0}, Y=l_{\infty}$ and $A$ be the summing basis of $c_{0}$. It is easy to check that $wk_{X}(A)=1$ and $wk_{Y}(A)=\frac{1}{2}$.

\begin{thm}\label{3.201}
Let $T:X\rightarrow L_{1}(\mu)(\mu$ finite measure) be an operator. Then
\item[(1)]$SCS_{l_{1}}(T)\leq \omega(T)=wk_{L_{1}}(T)\leq 16SCS_{l_{1}}(T);$
\item[(2)]$wk_{L_{1}}(T)\leq SS_{l_{1}}(T)\leq \sqrt[3]{2\|T\|^{2}wk_{L_{1}}(T)}.$
\end{thm}
\begin{proof}
(1) is a combination of \cite[Theorem 4.5]{LCC}, Theorem \ref{3.13} and inequality (\ref{206}).

\noindent (2). Step 1. $SS_{l_{1}}(T)\leq \sqrt[3]{2\|T\|^{2}wk_{L_{1}}(T)}.$

Fix any $0<c<SS_{l_{1}}(T).$ Then there is an operator $R:l_{1}\rightarrow X$ with $\|R\|=1$ such that $\|TRz\|\geq c\|z\|$ for all $z\in l_{1}$.
Let $\epsilon>wk_{L_{1}}(TB_{X})$ be arbitrary. Set $M=R(l_{1})$. Then $\|Tm\|\geq c\|m\|$ for all $m\in M$. This yields $B_{TM}\subseteq
\frac{1}{c}TB_{M}$
and hence $$wk_{L_{1}}(B_{TM})\leq \frac{1}{c}wk_{L_{1}}(TB_{M})\leq \frac{1}{c}wk_{L_{1}}(TB_{X})<\frac{\epsilon}{c}.$$
By Lemma \ref{5.4}, $wk_{TM}(B_{TM})\leq \frac{2\epsilon}{c}.$  It follows from the definition of quantity $wk(\cdot)$ that
$$wk_{M}(B_{M})\leq \|T|_{M}\|\|(T|_{M})^{-1}\|wk_{TM}(B_{TM})\leq \|T\|\frac{2\epsilon}{c^{2}}.$$ This implies
$$wk_{l_{1}}(B_{l_{1}})\leq wk_{M}(B_{M})\|R\|\|R^{-1}\|\leq \frac{2\epsilon \|T\|^{2}}{c^{3}}.$$
It follows from \cite[Proposition 7.3]{KKS} that $wk_{l_{1}}(B_{l_{1}})=1$. Thus $c^{3}\leq 2\epsilon \|T\|^{2}$.
Since $\epsilon>wk_{L_{1}}(TB_{X})$ is arbitrary, we get $c\leq \sqrt[3]{2\|T\|^{2}wk_{L_{1}}(T)}.$

Step 2. $wk_{L_{1}}(T)\leq SS_{l_{1}}(T)$.

We use the technique of \cite[Theorem 5.2.9]{AK}. Let $K=TB_{X}$. Let $0<c<wk_{L_{1}}(T)=wk_{L_{1}}(K)$. By \cite[Proposition 7.1]{KKS}, there are a
sequence $(f_{k})_{k}$ in $K$ and a sequence $(E_{k})_{k}$ of measurable subsets with $\lim_{k\rightarrow \infty}\mu(E_{k})=0$ such that
$\int_{E_{k}}|f_{k}|d\mu>c$ for all $k$. By \cite[Lemma 5.2.8]{AK}, by passing to subsequences if necessary, we obtain a sequence
$(A_{k})_{k}$ of pairwise disjoint measurable subsets such that $(f_{k}\chi_{B_{k}})_{k}$ is uniformly integrable, where $B_{k}=\Omega \setminus
A_{k}(k\in \mathbb{N})$. Since $(f_{k}\chi_{B_{k}})_{k}$ is uniformly integrable and $\lim_{k\rightarrow \infty}\mu(E_{k})=0$, we get
$\lim_{k\rightarrow \infty}\int_{E_{k}\cap B_{k}}|f_{k}|d\mu=0$. Let $\delta>0$ be arbitrary. By passing to subsequences again, we may assume that
$\int_{E_{k}\cap A_{k}}|f_{k}|d\mu>c-\delta$ for all $k$. Let $\alpha_{k}= \int_{A_{k}}|f_{k}|d\mu$ and $g_{k}=\alpha_{k}^{-1}f_{k}\chi_{A_{k}}$. For
each $k$, define $h_{k}\in L_{\infty}(\mu)$ by
\[h_{k}(\omega)= \left\{ \begin{array}
                    {r@{\quad,\quad}l}

 \frac{\overline{g_{k}(\omega)}}{|g_{k}(\omega)|} & |g_{k}(\omega)|>0\\ 0 & otherwise
 \end{array} \right. \]
Since the sequence $(A_{k})_{k}$ is pairwise disjoint, $(h_{k})_{k}$ and $(g_{k})_{k}$ are biorthogonal. Since $(f_{k}\chi_{B_{k}})_{k}$ is uniformly
integrable and $\lim_{k\rightarrow \infty}\mu(A_{k})=0$, we may assume that, by passing to subsequences, $\int_{A_{n}\cap
B_{m}}|f_{m}|d\mu<\frac{\delta}{2^{n}}$ for all $m,n.$ Define operators $$S:L_{1}\rightarrow l_{1}, f=(<h_{n},f>)_{n}, \quad f\in L_{1},$$ and
$$R:l_{1}\rightarrow L_{1}, (b_{k})_{k}\rightarrow \sum_{k=1}^{\infty}b_{k}\alpha_{k}^{-1}f_{k}, \quad (b_{k})_{k}\in l_{1}.$$
Then $\|S\|\leq 1$ and $\|R\|\leq \frac{\|T\|}{c-\delta}$. Since $(h_{k})_{k}$ and $(g_{k})_{k}$ are biorthogonal, we get $$
SRe_{k}-e_{k}=(\alpha_{k}^{-1}\int_{A_{n}\cap B_{k}}f_{k}h_{n}d\mu)_{n}, \quad (k=1,2,...).$$  Moreover,
\begin{align*}
\|SRe_{k}-e_{k}\|&=\alpha_{k}^{-1}\sum_{n=1}^{\infty}|\int_{A_{n}\cap B_{k}}f_{k}h_{n}d\mu|\\
&\leq \alpha_{k}^{-1}\sum_{n=1}^{\infty}\int_{A_{n}\cap B_{k}}|f_{k}|d\mu\\
&\leq \alpha_{k}^{-1}\sum_{n=1}^{\infty}\frac{\delta}{2^{n}}\\
&\leq \frac{\delta}{c-\delta}.\\
\end{align*}
This implies $\|SR-I_{l_{1}}\|\leq \frac{\delta}{c-\delta}$ and $(SR)^{-1}$ exists. Let $U=(SR)^{-1}$. Then $\|U\|\leq \frac{c-\delta}{c-2\delta}$.
For all scalars $b_{1},b_{2},...,b_{n}$ and all $n\in \mathbb{N}$, we have
\begin{align*}
\sum_{k=1}^{n}|b_{k}|&=\|USR(\sum_{k=1}^{n}b_{k}e_{k})\|\\
&=\|US(\sum_{k=1}^{n}b_{k}\alpha_{k}^{-1}f_{k})\|\\
&\leq \frac{c-\delta}{c-2\delta}\|\sum_{k=1}^{n}b_{k}\alpha_{k}^{-1}f_{k}\|.\\
\end{align*}
Thus, for all scalars $b_{1},b_{2},...,b_{n}$ and all $n\in \mathbb{N}$, we get $$\|\sum_{k=1}^{n}b_{k}f_{k}\|\geq
(c-2\delta)\sum_{k=1}^{n}|b_{k}|.$$
Let $f_{k}=Tx_{k}(x_{k}\in B_{X})$ for each $k$. Define an operator $V:l_{1}\rightarrow X$ by $Ve_{k}=x_{k}(k\in \mathbb{N})$.
Finally, $$SS_{l_{1}}(T)\geq \|V\|^{-1}\|(TV)^{-1}\|^{-1}\geq c-2\delta.$$
Letting $\delta\rightarrow 0$, we get $SS_{l_{1}}(T)\geq c$. The proof is completed.
\end{proof}
Recall that the James space $J$ is the (real) Banach space of all sequences $(a_{n})_{n}$ of real numbers such that
$\lim_{n\rightarrow \infty}a_{n}=0$ and
$$\|(a_{n})_{n}\|_{qv}=\sup\{(\sum_{j=1}^{m}|a_{i_{j-1}}-a_{i_{j}}|^{2})^{\frac{1}{2}}:1\leq i_{0}<i_{1}<\cdots<i_{m}, m\in \mathbb{N}\}<\infty.$$
The sequence $(e_{n})_{n}$ of standard unit vectors forms a monotone shrinking basis for $J$. We denote the coefficient functionals of $(e_{n})_{n}$
by $(e^{*}_{n})_{n}$.
\begin{thm}\label{3.10}
Let $T:J\rightarrow J$ be an operator. Then
\item[(1)]$\frac{1}{2\sqrt{5}}\chi(T^{*})\leq SS_{l_{2}}(T)\leq SS(T)\leq 2\chi(T)$;
\item[(2)]$\frac{1}{4\sqrt{10}}\chi(T^{*})\leq SCS_{l_{2}}(T)\leq SCS(T)\leq 2\chi(T^{*}).$
\end{thm}
\begin{proof}
Let $0<c<\chi(T^{*})$. Let $\epsilon>0$ be arbitrary. By induction on finite-codimensional subspaces of $J$, we obtain a sequence of $(k_{n})_{n\geq
0}, k_{n}-k_{n-1}>1 (n=1,2,...)$ and a sequence $(x_{n})_{n}\in B_{J}$ such that $\|Tx_{n}\|_{qv}>c$ and $\|Tx_{n}-Tu_{n}\|_{qv}<\frac{\epsilon}{2^{n}}$ for each $n$, where
$u_{n}=\sum_{k=k_{n-1}+1}^{k_{n}-1}<e^{*}_{k},x_{n}>e_{k}(n=1,2,...)$. Since $(e_{n})_{n}$ is monotone and $(x_{n})_{n}\in B_{J}$, we get
$\|u_{n}\|_{qv}\leq 2$ for each $n$. It follows from \cite[Proposition 3.4.3]{AK} that
\begin{equation}\label{89}
\|\sum_{n=1}^{m}b_{n}u_{n}\|_{qv}\leq 2\sqrt{5}(\sum_{n=1}^{m}|b_{n}|^{2})^{\frac{1}{2}},
\end{equation}
for all scalars $b_{1},b_{2},...,b_{m}$ and all $m\in \mathbb{N}$.

By Theorem \ref{3.22} and \cite[Theorem 10]{CLL}, there exist a subsequence of $(u_{n})_{n}$, which is still denoted by $(u_{n})_{n}$, and
a skipped block basic sequence $(z_{n})_{n}$ of $(e_{n})_{n}$ such that
\begin{equation}\label{86}
(1-\epsilon)\|\sum_{n=1}^{m}b_{n}z_{n}\|_{qv}\leq \|\sum_{n=1}^{m}b_{n}Tu_{n}\|_{qv}\leq (1+\epsilon)\|\sum_{n=1}^{m}b_{n}z_{n}\|_{qv}.
\end{equation}
Moreover, $\overline{span}\{Tu_{n}:n\in \mathbb{N}\}$ is $2\sqrt{2}\frac{1+\epsilon}{1-\epsilon}$-complemented in $J$.

The right side of inequality (\ref{86}) yields that $\|z_{n}\|_{qv}\geq \frac{c-\epsilon}{1+\epsilon}$ for all $n$. The proof of \cite[Lemma 1]{HW}
implies that
\begin{equation}\label{87}
\|\sum_{n=1}^{m}b_{n}z_{n}\|_{qv}\geq \frac{c-\epsilon}{1+\epsilon}(\sum_{n=1}^{m}|b_{n}|^{2})^{\frac{1}{2}},
\end{equation}
for all scalars $b_{1},b_{2},...,b_{m}$ and all $m\in \mathbb{N}$.

Combining (\ref{86}) and (\ref{87}), we obtain
\begin{equation}\label{88}
\|\sum_{n=1}^{m}b_{n}Tu_{n}\|_{qv}\geq (1-\epsilon)\frac{c-\epsilon}{1+\epsilon}(\sum_{n=1}^{m}|b_{n}|^{2})^{\frac{1}{2}},
\end{equation}
for all scalars $b_{1},b_{2},...,b_{m}$ and all $m\in \mathbb{N}$.

\noindent (1). Define an operator $R:l_{2}\rightarrow J$ by $R((b_{n})_{n})=\sum_{n=1}^{\infty}b_{n}u_{n}, (b_{n})_{n}\in l_{2}.$ Thus (\ref{89}) gives that
$\|R\|\leq 2\sqrt{5}.$ We deduce from (\ref{88}) that
$$SS_{l_{2}}(T)\geq \|R\|^{-1}\|(TR)^{-1}\|^{-1}\geq \frac{1}{2\sqrt{5}}(1-\epsilon)\frac{c-\epsilon}{1+\epsilon}.$$
Letting $\epsilon\rightarrow 0$, we get $SS_{l_{2}}(T)\geq \frac{c}{2\sqrt{5}}$ and prove (1).

\noindent (2). Let $P$ be a projection from $J$ onto $\overline{span}\{Tu_{n}:n\in \mathbb{N}\}$ with $\|P\|\leq 2\sqrt{2}\frac{1+\epsilon}{1-\epsilon}$. Define an operator $S:\overline{span}\{Tu_{n}:n\in \mathbb{N}\}\rightarrow l_{2}$ by $$S(\sum_{n=1}^{\infty}b_{n}Tu_{n})=(b_{n})_{n}, \quad (b_{n})_{n}\in l_{2}$$
Inequalities (\ref{86}) and (\ref{88}) ensure that $S$ is well-defined and $\|S\|\leq \frac{1+\epsilon}{(1-\epsilon)(c-\epsilon)}.$ Let $R=SP:J\rightarrow l_{2}$. It follows from (\ref{89}) and the definition of $R$ that $B_{l_{2}}\subseteq 2\sqrt{5}RTB_{J}$. Thus
$$SCS_{l_{2}}(T)\geq \frac{\delta(RT)}{\|R\|}\geq \frac{1}{2\sqrt{5}}\frac{(1-\epsilon)^{2}(c-\epsilon)}{2\sqrt{2}(1+\epsilon)^{2}}.$$
Letting $\epsilon\rightarrow 0$, we get $SCS_{l_{2}}(T)\geq \frac{c}{4\sqrt{10}}$ and prove (2).

\end{proof}

\section{$\lambda$-subprojective spaces and $\lambda$-superprojective spaces}

\begin{defn}\cite{OS}
Let $\lambda\geq 1$. We say that a Banach space $X$ is $\lambda$-subprojective if every infinite-dimensional subspace of $X$ contains an
infinite-dimensional subspace that is $\lambda$-complemented in $X$.
\end{defn}

\begin{thm}\label{5.10}
Let $Y$ be $\lambda$-subprojective and $T:X\rightarrow Y$ an operator. Then $$SS(T)\leq \lambda SS(T^{*}).$$
\end{thm}
\begin{proof}
We may assume that $SS(T)>0$ and fix any $0<c<SS(T)$. Then there is an infinite-dimensional subspace $M$ of $X$ such that $\|Tx\|\geq c\|x\|$ for all
$x\in M$. Since $Y$ is $\lambda$-subprojective, there exist an infinite-dimensional subspace $N\subseteq TM$ and a projection $P$ from $Y$ onto $N$
with $\|P\|\leq \lambda$. We write $N$ as $TX_{0}, X_{0}\subseteq M$.
Then $P^{*}$ is a surjective isomorphism from $N^{*}$ onto $(Ker(P))^{\perp}$ and $P^{*}f|_{N}=f$ for each $f\in N^{*}$.

\noindent Claim: $\|T^{*}y^{*}\|\geq \frac{c}{\lambda}\|y^{*}\|$ for each $y^{*}\in (Ker(P))^{\perp}$.

Indeed, for $y^{*}\in (Ker(P))^{\perp}, y^{*}=P^{*}f, f\in N^{*}$, we have
\begin{align*}
\|T^{*}y^{*}\|=\|T^{*}P^{*}f\|&=\sup_{x\in B_{X}}|<T^{*}P^{*}f,x>|\\
&\geq \sup_{x\in B_{X_{0}}}|<P^{*}f,Tx>|\\
&=\sup_{x\in B_{X_{0}}}|<f,Tx>|\\
&\geq \sup_{y\in c\cdot B_{N}}|<f,y>|\\
&=c\|f\|\geq \frac{c}{\lambda}\|P^{*}f\|=\frac{c}{\lambda}\|y^{*}\|.
\end{align*}
By Claim, we get $SS(T^{*})\geq \frac{c}{\lambda}$. The proof is completed due to the arbitrariness of $c$.

\end{proof}

We list some of known $\lambda$-subprojective spaces:

$\bullet$ $l_{p}(1\leq p<\infty),c_{0}$ are $\lambda$-subproiective for any $\lambda>1$.

$\bullet$ $L_{p}(2<p<\infty)$ is $C_{p}$-subprojective, where the constant $C_{p}$ depends only on $p$.

More precisely, if $X$ is an infinite-dimensional subspace of $L_{p}(2<p<\infty)$, then either for every $\epsilon>0$, $X$ contains an
infinite-dimensional subspace that is $(1+\epsilon)\gamma_{p}$-complemented in $L_{p}$, where $\gamma_{p}$ is the norm of a symmetric Gaussian random
variable(see \cite{HOS}), or for every $\epsilon>0$, $X$ contains an infinite-dimensional subspace that is $(1+\epsilon)$-complemented in $L_{p}$ (see
\cite{KP}).

$\bullet$ The Tsirelson space $T$ is $\lambda$-subproiective for any $\lambda>54$(see \cite{CS}).

$\bullet$ The Lorentz sequence spaces $d(w,p)$ are $\lambda$-subproiective for any $\lambda>1$(see \cite[Proposition 4.e.3]{LT}).

$\bullet$ Every $c_{0}$-saturated separable space is $2\lambda$-subproiective for any $\lambda>1$.

Recall that a Banach space $X$ is said to be $c_{0}$-saturated if every infinite-dimensional subspace of $X$ contains an isomorphic copy of $c_{0}$.
For example, $C(K)$($K$ countable compact space)(see \cite{FHHMPZ}), $C(\alpha)$($\alpha$ countable ordinal)(see \cite{PS}), the quotient of the
Schreier space(see \cite{O}), the injective tensor product of $JH$ and $JH$ ($JH$ the James Hagler space)(see \cite{L}),
$c_{0}\widehat{\otimes}c_{0}$(see \cite{GS}), the projective tensor product of two $C(K)$-spaces ($K$ infinite countable compact metric space)(see
\cite{GS}), $C(\alpha)\widehat{\otimes}C(\beta)(\omega<\alpha,\beta<\omega_{1})$(see \cite{GS}) are all $c_{0}$-saturated.

$\bullet$ The James space $J$ with the quadratic variation norm is $\lambda$-subproiective for any $\lambda>2\sqrt{2}$(see \cite{CLL}).

\begin{cor}
Let $1<p<\infty$. Then $$SS(T)=SS(T^{*}),$$
for any operator $T:l_{p}\rightarrow l_{p}$.
\end{cor}

\begin{defn}
Let $\lambda\geq 1$. We say that a Banach space $X$ is $\lambda$-superprojective if, given any subspace $M$ of $X$ with infinite codimension, there is
a subspace $N$ containing $M$ such that $N$ has infinite codimension and is $\lambda$-complemented in $X$.
\end{defn}

\begin{thm}
Let $X$ be a reflexive Banach space and $\lambda\geq 1$. Then
\item[(1)]If $X$ is $\lambda$-subprojective, then $X^{*}$ is $(1+\lambda)$-superprojective;
\item[(2)]If $X^{*}$ is $\lambda$-superprojective, then $X$ is $(1+\lambda)$-subprojective.
\end{thm}
\begin{proof}
\noindent(1). Let $M$ be a subspace of $X^{*}$ with infinite codimension. Since $X$ is reflexive, $^\perp M$ is an infinite-dimensional subspace of
$X$. By the assumption, there exist an infinite-dimensional subspace $N$ of $^\perp M$ and a projection $P$ from $X$ onto $N$ with $\|P\|\leq
\lambda$. Then $I_{X^{*}}-P^{*}$ is a projection from $X^{*}$ onto $N^{\perp}\supseteq M$. Hence $N^{\perp}$ is $(1+\lambda)$-complemented in $X^{*}$,
has infinite codimension and contains $M$.

\noindent(2). Let $M$ be an infinite-dimensional subspace of $X$. Since $X^{*}$ is $\lambda$-superprojective, there exist a subspace $N$ of $X^{*}$
with infinite codimension, $N\supseteq M^{\perp}$ and a projection $P$ from $X^{*}$ onto $N$ with $\|P\|\leq \lambda$. Since $X$ is reflexive,
$P=Q^{*}$, where $Q:X\rightarrow X$ is a projection. It is easy to verify that $I_{X}-Q$ is a projection from $X$ onto $^\perp N$. Hence $^\perp N$ is
$(1+\lambda)$-complemented in $X$, infinite-dimensional and is contained in $M$.

\end{proof}

\begin{thm}\label{5.5}
Let $X$ be $\lambda$-superprojective and $T:X\rightarrow Y$ an operator. Then $$SCS(T)\leq (1+\lambda) SCS(T^{*}).$$
\end{thm}
\begin{proof}
Suppose that $SCS(T)>0$. Given arbitrary $0<c<SCS(T)$. Then there is an infinite-codimensional subspace $N$ of $Y$ such that $c\cdot B_{Y/N}\subseteq
Q_{N}TB_{X}$. Since $Ker(Q_{N}T)=T^{-1}(N)$ is infinite-codimensional, there exist an infinite-codimensional subspace $M$ of $X$ such that
$T^{-1}(N)\subseteq M$ and a projection $P$ from $X$ onto $M$ with $\|P\|\leq \lambda$. Since $M$ is infinite-codimensional, $P^{*}X^{*}$ is also
infinite-codimensional.

Claim: $B_{X^{*}/P^{*}X^{*}}\subseteq \frac{(1+\epsilon)(1+\lambda)}{c}\cdot Q_{P^{*}X^{*}}T^{*}B_{Y^{*}}$ for every $\epsilon>0$.

\noindent Let us fix any $\varphi\in B_{X^{*}/P^{*}X^{*}}$. Then there exists $x^{*}\in X^{*}$ such that $Q_{P^{*}X^{*}}x^{*}=\varphi$ and
$\|x^{*}\|<\|\varphi\|+\epsilon$. We define $y^{*}\in Y^{*}$ by
\begin{center}
$<y^{*},y>=<x^{*},(I-P)x>$, $y\in Y$, where $y-Tx\in N$.
\end{center}
Then $y^{*}$ is well-defined, linear and $\|y^{*}\|\leq \frac{(1+\epsilon)(1+\lambda)}{c}$. In fact, for $y\in Y$, there exists $x\in X$ such that
$Q_{N}y=Q_{N}Tx$ and $\|x\|\leq \frac{\|Q_{N}y\|}{c}\leq \frac{\|y\|}{c}$. Thus $y-Tx\in N$ and then
\begin{align*}
|<y^{*},y>|&=|<x^{*},(I-P)x>|\\
&\leq \|x^{*}\|(1+\lambda)\|x\|\\
&\leq (1+\epsilon)(1+\lambda)\|x\|\\
&\leq (1+\epsilon)(1+\lambda)\frac{\|y\|}{c}.\\
\end{align*}
Thus $\|y^{*}\|\leq \frac{(1+\epsilon)(1+\lambda)}{c}$.

It remains to prove that $\varphi=Q_{P^{*}X^{*}}x^{*}=Q_{P^{*}X^{*}}T^{*}y^{*}$. Indeed, by the definition of $y^{*}$, we get
\begin{center}
$<y^{*},Tx>=<x^{*},(I-P)x>$, for all $x\in X$.
\end{center}
This yields that $T^{*}y^{*}=(I_{X^{*}}-P^{*})x^{*}$. Thus $(I_{X^{*}}-P^{*})T^{*}y^{*}=(I_{X^{*}}-P^{*})x^{*}$, namely,
$Q_{P^{*}X^{*}}x^{*}=Q_{P^{*}X^{*}}T^{*}y^{*}$.

By Claim, we get $\frac{c}{(1+\epsilon)(1+\lambda)}\leq SCS(T^{*})$ for every $\epsilon>0$. This means that $\frac{c}{(1+\lambda)}\leq SCS(T^{*})$.
By the arbitrariness of $c$, we get the conclusion.

\end{proof}

{\bf Acknowledgements.} This work is done during the second author's visit to Department of Mathematics, Texas A\&M University. We are grateful to Professor W. B. Johnson for helpful comments.


\begin{thebibliography}{MMM}
\frenchspacing


\bibitem[1]{AK}
F. Albiac and N. J. Kalton,
{\em Topics in Banach space theory}, Springer, 2005.

\bibitem[2]{Als}
D. Alspach, {\em Good $l_{2}$-subspaces of $L_{p},p>2$}, Banach J. Math. Anal. {\bf 3}(2009), 49-54.


\bibitem[3]{AO}
D. Alspach and O. Edell, {\em $L_{p}$ spaces}, In: Handbook of the Geometry of Banach Spaces, Vol. 1, W. B. Johnson and J. Lindenstrauss, eds,
North-Holland, Amsterdam, 123-159(2001).



\bibitem[4]{A1}
A. Andrew, {\em James' quasi-reflexive space is not isomorphic to any subspace of its dual}, Israel J. Math. {\bf 38}(1981),276-282.


\bibitem[5]{AC}
C. Angosto and B. Cascales, {\em Measures of weak non-compactness in Bananch spaces}, Topology Appl. {\bf 156}(2009), 1412-1421.

\bibitem[6]{Ben}
H. Bendov\'{a}, {\em Quantitative Grothendieck property}, J. Math. Anal. Appl. {\bf 412}(2014), 1097-1104.


\bibitem[7]{BKS}
H. Bendov\'{a}, O. F. K. Kalenda and J. Spurn\'{y}, {\em Quantification of the Banach-Saks property}, J. Funct. Anal. {\bf 268}(2015), 1733-1754.


\bibitem[8]{BP}
C. Bessaga and A. Pe{\l}czy\'{n}ski, {\em On bases and unconditional convergence of series in Banach spaces}, Studia Math.{\bf 17}(1958),151-164.


\bibitem[9]{BD}J. Bourgain and F. Delbaen,
{\em A class of special $\mathcal{L}_{\infty}$ spaces}, Acta Math. {\bf 145}(1980),no.3--4, 155--176.

\bibitem[10]{CLL}
P. G. Casazza, Bor-Luh Lin and R. H. Lohman, {\em On James' quasi-reflexive Banach spaces}, Proc. Amer. Math. Soc. {\bf 67}(1977), 265-271.

\bibitem[11]{CS}
P. G. Casazza and T. J. Shura, {\em Tsirelson's space}, Lecture Notes in Mathematics, 1363, Springer-Verlag, Berlin, 1989.

\bibitem[12]{CKS}
B. Cascales, O. F. K. Kalenda and J. Spurn\'{y}, {\em A quantitative version of James' compactness theorem}, Proc. Edinburgh Math. Soc. {\bf 55}(2012), 369-386.

\bibitem[13]{Cas}
J. M. F. Castillo, {\em On Banach spaces $X$ such that $L(L_{p},X)=K(L_{p},X)$}, Extracta Math. {\bf 10}(1995), 27-36.


\bibitem[14]{CS2}
J. M. F. Castillo and F. S\'{a}nchez, {\em Dunford-Pettis-like properties of continuous vector function spaces},
Revista Matematica.{\bf 6}(1993), 43-59.

\bibitem[15]{CCL}
D. Chen, J. Alejandro Ch\'{a}vez-Dom\'{i}nguez and L. Li, {\em Unconditionally $p$-converging operators and Dunford-Pettis Property of order $p$},
arXiv: 1607.02161.



\bibitem[16]{D}
J. Diestel,
{\em Sequences and Series in Banach Spaces}, Springer, New York, 1984.


\bibitem[17]{FHHMPZ}
M. Fabian, P. Habala, P. Hajek, V. Montesinos, J. Pelant and V. Zizler, {\em Functional analysis and infinite-dimensional geometry}, Springer, New
York, 2001.

\bibitem[18]{FHMZ}
M. Fabian, P. H\'{a}jek, V. Montesinos and V. Zizler, {\em A quantitative version of Krein's theorem}, Rev. Mat. Iberoamer. {\bf 21}(2005), 237-248.


\bibitem[19]{FGJ}
T. Figiel, N. Ghoussoub and W. B. Johnson, {\em On the structure of non-weakly compact operators on Banach lattices}, Math. Ann. {\bf 257}(1981),
317-334.

\bibitem[20]{GS}
E. M. Galego and C. Samuel, {\em The subprojectivity of the projective tensor product of two $C(K)$ spaces with $|K|=\aleph_{0}$}, Proc. Amer. Math.
Soc.{\bf 144}(2016),2611-2617.


\bibitem[21]{GM}
L. S. Gol'den\v{s}te\v{i}n and A. S. Markus, {\em On the measure of non-compactness of bounded sets and of linear operators}, in: Studies in Algebra
and Math. Anal., Izdat. Karta Moldovenjaske, Kishinev, 1965, pp.45-54(in Russian).

\bibitem[22]{HOS}
R. Haydon, E. Odell and T. Schlumprecht, {\em Small subspaces of $L_{p}$}, Ann. Math. {\bf 173}(2011), 169-209.



\bibitem[23]{HW}
R. Herman and R. Whitley, {\em An example concerning reflexivity}, Studia Math. {\bf 28}(1967), 289-294.


\bibitem[24]{JZ}
W.B.Johnson and M.Zippin, {\em Separable $L_{1}$-preduals are quotients of $C(\Delta)$}, Israel J. Math. {\bf 16}(1973), 198-202.



\bibitem[25]{KKS}
M. Ka\v{c}ena, O. F. K. Kalenda and J. Spurn\'{y}, {\em Quantitative Dunford-Pettis property}, Adv. Math.
{\bf 234}
(2013), 488-527.

\bibitem[26]{KP}
M. I. Kadec and A. Pe{\l}czy\'{n}ski, {\em Bases, lacunary sequences and complemented subspaces in the spaces $L_{p}$}, Studia Math.{\bf 21}
(1962),161-176.

\bibitem[27]{KPS}
O. F. K. Kalenda, H. Pfitzner and J. Spurn\'{y}, {\em On quantification of weak sequential completeness},
J. Funct. Anal. {\bf 260}(2011), 2986-2996.

\bibitem[28]{KS}
O. F. K. Kalenda and J. Spurn\'{y}, {\em Quantification of the reciprocal Dunford-Pettis property}, Studia Math. {\bf 210}(2012), 261-278.

\bibitem[29]{KS1}
O. F. K. Kalenda and J. Spurn\'{y}, {\em On a difference between quantitative weak sequential completeness and the quantitative Schur property}, Proc. Amer. Math. Soc. {\bf 140}(2012), 3435-3444.


\bibitem[30]{Ka}
T. Kato, {\em Perturbation theory for nullity deficiency and other quantities of linear operators}, {\bf 6}(1958), 273-322.



\bibitem[31]{Kr}
H. Kruli\v{s}ov\'{a}, {\em Quantification of Pe{\l}czy\'{n}ski's property (V)}, arXiv:1509.06610v1.

\bibitem[32]{Kr1}
H. Kruli\v{s}ov\'{a}, {\em $C^{*}$-algebras have a quantitative version of Pe{\l}czy\'{n}ski's property (V)}, arXiv:1605.04900v2.

\bibitem[33]{L}
Denny H. Leung, {\em Some stabilities of $c_{0}$-saturated spaces}, Math. Proc. Cambridge Philos. Soc. {\bf 118}(1995),287-301.

\bibitem[34]{LCC}
L.Li, D. Chen and J. Alejandro Ch\'{a}vez-Dom\'{i}nguez, {\em Pe{\l}czy\'{n}ski's property ($V^{*}$) of order $p$ and its quantification},
arXiv:1607.02163.

\bibitem[35]{LT} J. Lindenstrauss and L. Tzafriri,
{\em Classical Banach Spaces I, Sequence Spaces}, Springer,
Berlin, 1977.

\bibitem[36]{O}
E. Odell, {\em On quotients of Banach spaces having shrinking unconditional bases}, Illinois J. Math.{\bf 36}(1992),681-695.

\bibitem[37]{OS}
T. Oikhberg and E. Spinu, {\em Subprojective Banach spaces}, J. Math. Anal. Appl.{\bf 424}(2015),613-635.


\bibitem[38]{P1}
A. Pe{\l}czy\'{n}ski, {\em On strictly singular and strictly cosingular operators. I. Strictly singular and strictly cosingular operators in
$C(S)$-spaces}, Bull. Acad. Polon. Sci. {\bf 13}(1965), 31-36.

\bibitem[39]{P2}
A. Pe{\l}czy\'{n}ski, {\em On strictly singular and strictly cosingular operators. II. Strictly singular and strictly cosingular operators in
$L(\nu)$-spaces}, Bull. Acad. Polon. Sci. {\bf 13}(1965), 37-41.


\bibitem[40]{PS}
A. Pe{\l}czy\'{n}ski and Z. Semadeni, {\em Spaces of continuous functions. III. Spaces $C(\Omega)$ for $\Omega$ without perfect subsets}, Studia
Math.{\bf 18}(1959),211-222.


\bibitem[41]{W}
R. J. Whitley, {\em Strictly singular operators and their conjugates}, Trans. Amer. Math. Soc. {\bf 113}(1964), 252-261.



\end{thebibliography}
\end{document}